\RequirePackage{etoolbox}
\csdef{input@path}{%
 {sty/},%
 {bib/}%
}
\documentclass[numbers,compress]{vmsta2}

\volume{0}
\issue{0}
\pubyear{2026}
\articletype{research-article}

\newtheorem{theorem}{Theorem}[section]
\newtheorem{proposition}[theorem]{Proposition}
\newtheorem{lemma}[theorem]{Lemma}
\newtheorem{corollary}[theorem]{Corollary}
\theoremstyle{definition}
\newtheorem{definition}[theorem]{Definition}

\theoremstyle{remark}
\newtheorem{remark}[theorem]{Remark}

\newcommand{\R}{\mathbb R}
\newcommand{\E}{\mathbb E}
\newcommand{\Pp}{\mathbb P}
\newcommand{\e}{\mathrm e}
\newcommand{\ip}[2]{\langle #1,#2\rangle}
\newcommand{\Law}{\mathcal L}
\newcommand{\supp}{\operatorname{supp}}
\newcommand{\interior}{\operatorname{int}}
\newcommand{\cone}{\operatorname{cone}}
\newcommand{\conv}{\operatorname{conv}}

\begin{document}

\begin{frontmatter}

\pretitle{Research Article}

\title{Minimal Radial Sub-Gamma Envelopes for Infinitely Divisible Random Vectors}
\runtitle{Minimal radial sub-gamma envelopes}

\begin{aug}
\author[a]{\inits{Y.}\fnms{Yichuan}~\snm{Chen}\thanksref{cor1}\ead[label=e1]{smallboychen@gmail.com}}
\author[a]{\inits{X.}\fnms{Xin}~\snm{Wang}\ead[label=e2]{wxcaup@zstu.edu.cn}}
\thankstext[type=corresp,id=cor1]{Corresponding author.}
\address[a]{\institution{School of Art and Design, Zhejiang Sci-Tech University}, Hangzhou 310018, \cny{China}}
\runauthor{Y. Chen and X. Wang}
\end{aug}

\begin{abstract}
Let \(X\) be a centered infinitely divisible random vector with finite second moment and covariance matrix \(\Sigma\). We define the radial pole \(C_X(t)\) as the smallest scale in a right sub-gamma bound for \(\ip{t}{X}\) whose quadratic proxy is fixed at the true variance \(t^\top\Sigma t\). A canonical directional measure and an independent \(\mathrm{Beta}(1,2)\) multiplier give an exact variational formula for \(C_X\). The resulting extended-valued function is positive homogeneous and is pointwise least among all homogeneous denominators compatible with the covariance quadratic form. We prove linear-map, convolution, and L\'evy-time rules, and show that the full family of directional remainders determines the law of \(X\). Geometrically, \(C_X\) lies between the Minkowski functional of the moment-generating-function domain and one third of the positive support function of the L\'evy measure; the upper constant is sharp, and the zero set is a polar cone. The pole need not be subadditive. It is continuous on the sphere under global exponential moments and positive-definite covariance, whereas finite variance alone permits a jump from zero to infinity in nearby directions. For additive gamma-ray models, \(C_X\) equals the domain gauge and has a finite-polytope formula.
\end{abstract}



\end{frontmatter}

\section{Introduction}\label{sec:introduction}

A centered random variable \(Y\) is right sub-gamma with variance factor \(v>0\) and scale \(c\geq0\) if
\begin{equation}\label{eq:scalar-subgamma}
 \log\E\e^{sY}\leq \frac{v s^2}{2(1-cs)},
 \qquad 0\leq s<1/c,
\end{equation}
where \(1/0=\infty\). This form gives the Bernstein tail
\(\Pp\{Y\geq\sqrt{2vx}+cx\}\leq\e^{-x}\) and is standard in concentration theory \citep{boucheron2013}. The pair \((v,c)\) is not intrinsic: increasing either coordinate preserves the inequality. Fixing \(v=\operatorname{Var}(Y)\) removes that freedom at quadratic order and leaves a global question about the smallest admissible pole \(c\). Variance-matched scale optimization has been carried out for particular scalar families, including beta distributions \citep{skorski2023}. It is different from optimizing the variance proxy in a Poisson-type envelope \citep{leskela2026}.

For a random vector, the usual formulation replaces \(sY\) by \(\ip{t}{X}\) and controls all \(t\) with one Euclidean pair. Zhang, Cheng, and Reeves \citep{zhang2021}, for example, call \(X\in\R^d\) sub-gamma when its centered moment generating function is bounded by an expression of the form
\[
 \frac{v\|t\|^2}{2(1-b\|t\|)}.
\]
That definition is useful when an isotropic constant is needed. It does not retain the covariance quadratic form \(t^\top\Sigma t\), and a single value of \(b\) cannot record that the positive exponential-moment boundary may change sharply with direction.

Infinitely divisible vectors make this loss visible. Their projections remain infinitely divisible, but their positive jumps, Gaussian variance, and moment-generating-function domains depend on direction. General deviation inequalities for functions of infinitely divisible vectors were established by Houdr\'e \citep{houdre2002} and Paulauskas \citep{paulauskas2002}. Norm concentration under component assumptions was studied by Houdr\'e, Marchal, and Reynaud-Bouret \citep{houdre2008}, while Kontoyiannis and Madiman \citep{kontoyiannis2006} treated compound Poisson concentration. These bounds solve different problems: they do not identify the least denominator in \eqref{eq:scalar-subgamma} after the variance of every projection has been fixed.

We study that denominator as a function of the test vector. For centered \(X\in\R^d\) with covariance \(\Sigma\), let \(C_X(t)\) be the smallest scale for which \eqref{eq:scalar-subgamma} holds with \(Y=\ip{t}{X}\) and \(v=t^\top\Sigma t\). The map is allowed to take the value infinity. Its positive homogeneity makes it a radial object, but not necessarily a convex gauge. The term \emph{radial pole} will distinguish \(C_X\) from the Minkowski functional of a convex set.

The analysis begins with the canonical measure of each projection, an established representation in the theory of infinitely divisible laws \citep{sato2013,steutel2004} and a separate object of statistical estimation \citep{watteel2003}. Multiplication by an independent \(\mathrm{Beta}(1,2)\) variable turns the cumulant into the true variance times an ordinary moment generating function. This identity yields an exact variational formula for \(C_X(t)\), including directions with no positive exponential moment. It also makes the passage from scalar projections to a multivariate minimum transparent: \(C_X\) is the pointwise least positive-homogeneous denominator compatible with the covariance quadratic form.

The radial minimum has a geometric structure that is absent from a direction-by-direction statement. Its lower bound is the Minkowski functional of the interior of the moment-generating-function domain. Its upper bound is one third of the positive support function of the L\'evy measure, with the factor \(1/3\) attained by a centered Poisson law. The zero set is the polar of the closed cone generated by the jump support. We also obtain transformation and convolution rules, and an imaginary-argument version of the factorization shows that the directional remainder family determines the law. These results use standard L\'evy and convex-analytic foundations \citep{sato2013,rockafellar1970}; series representations of infinitely divisible vectors provide related geometric background \citep{rosinski1990}.

Two boundary phenomena are exact. A two-atom compound Poisson vector shows that \(C_X\) can fail subadditivity even when the covariance is positive definite. On the other hand, if the cumulant is finite on all of \(\R^d\) and \(\Sigma\) is positive definite, then \(C_X\) is continuous on the unit sphere. A finite-variance compound Poisson example shows that the stronger moment assumption cannot simply be dropped: the pole is zero in one direction and infinite along directions converging to it. Finally, for a Gaussian vector plus finitely many independent gamma rays, the pole equals the moment-domain gauge and is the support function of a finite polytope.

These strands leave a specific gap. Isotropic vector bounds provide one Euclidean denominator, while scalar optimization produces a number for a fixed law or direction. Neither formulation gives a covariance-matched homogeneous minimum over all projections or links that minimum simultaneously to the MGF domain, L\'evy jump geometry, and exact transformation laws. This is the problem addressed here.

This paper continues two scalar manuscripts in the same research series. The spectrally positive problem came first \citep{chen2026spectral}, followed by the general two-sided problem \citep{chenwang2026directional}. Passing to vectors introduces the homogeneous minimum, domain and support geometry, nonconvexity, the regularity boundary, and the gamma-ray formula. We include the directional factorization needed below so that the argument is complete.

Section~\ref{sec:factorization} constructs the pole and proves radial minimality. Section~\ref{sec:geometry} develops its structural and geometric properties. Section~\ref{sec:regularity} gives the non-subadditivity and regularity results. Section~\ref{sec:gamma} treats gamma-ray models, and Section~\ref{sec:discussion} closes with the implications and remaining questions.

\section{Directional factorization and the minimal radial envelope}\label{sec:factorization}

Let \(X\) be a centered infinitely divisible random vector in \(\R^d\) with finite second moment. Write \(A\) for the covariance matrix of its Gaussian component and \(\nu\) for its L\'evy measure. With the drift chosen to center \(X\), the cumulant is
\begin{equation}\label{eq:LK}
 K_X(t):=\log\E\e^{\ip{t}{X}}
 =\frac12t^\top A t+
 \int_{\R^d}\bigl(\e^{\ip{t}{x}}-1-\ip{t}{x}\bigr)\,\nu(dx),
\end{equation}
at every point where the positive part is integrable. Its covariance matrix is
\begin{equation}\label{eq:covariance}
 \Sigma=A+\int_{\R^d}xx^\top\,\nu(dx),
 \qquad V_X(t):=t^\top\Sigma t.
\end{equation}
The integral form in \eqref{eq:LK} is legitimate under the second-moment assumption; see, for example, Sato \citep{sato2013}.

Fix \(t\in\R^d\) with \(V_X(t)>0\). Define a finite measure \(H_t\) on \(\R\) by
\begin{equation}\label{eq:canonical-measure}
 H_t(B)=(t^\top A t)\mathbf 1_{\{0\in B\}}
 +\int_{\R^d}\mathbf 1_{\{\ip{t}{x}\in B\}}\ip{t}{x}^2\,\nu(dx).
\end{equation}
Its total mass is \(V_X(t)\). Let \(Y_t\) have distribution \(H_t/V_X(t)\), let \(B\) be independent of \(Y_t\) with density \(2(1-b)\mathbf 1_{(0,1)}(b)\), and set
\begin{equation}\label{eq:remainder}
 R_t=BY_t.
\end{equation}
The variable \(R_t\) may have either sign.

\begin{lemma}[Beta identity]\label{lem:beta}
For every real or purely imaginary \(z\), with the value at zero understood by continuity,
\begin{equation}\label{eq:beta-identity}
 \E\e^{zB}=\frac{2(\e^z-1-z)}{z^2}.
\end{equation}
\end{lemma}

\begin{proof}
Direct integration gives
\[
 2\int_0^1(1-b)\e^{zb}\,db
 =\frac{2(\e^z-1-z)}{z^2}.
\]
The right-hand side tends to one as \(z\to0\).
\end{proof}

\begin{theorem}[Canonical directional factorization]\label{thm:factorization}
For \(s\geq0\), with extended values allowed,
\begin{equation}\label{eq:factorization}
 K_X(st)=\frac{V_X(t)s^2}{2}M_t(s),
 \qquad M_t(s):=\E\e^{sR_t}.
\end{equation}
For \(r\in\R\), the characteristic exponent satisfies
\begin{equation}\label{eq:imaginary-factorization}
 \Psi_X(rt):=\log\E\e^{ir\ip{t}{X}}
 =-\frac{V_X(t)r^2}{2}\,\varphi_t(r),
 \qquad \varphi_t(r):=\E\e^{irR_t}.
\end{equation}
\end{theorem}

\begin{proof}
Put \(y=\ip{t}{x}\). Lemma~\ref{lem:beta} and \eqref{eq:canonical-measure} give
\begin{align*}
 M_t(s)
 &=\frac{1}{V_X(t)}\int_{\R}\E\e^{sBy}\,H_t(dy)\\
 &=\frac{1}{V_X(t)}\left[t^\top A t+
 \frac{2}{s^2}\int_{\R^d}(\e^{s\ip{t}{x}}-1-s\ip{t}{x})\,\nu(dx)\right].
\end{align*}
Multiplication by \(V_X(t)s^2/2\) proves \eqref{eq:factorization}. The same calculation with \(s=ir\) is absolutely justified by the second-moment form of the L\'evy--Khintchine formula and gives \eqref{eq:imaginary-factorization}.
\end{proof}

We now define the scale without assuming that a positive exponential moment exists.

\begin{definition}[Directional pole]\label{def:pole}
If \(V_X(t)>0\), let \(C_X(t)\in[0,\infty]\) be the infimum of all \(c\geq0\) such that
\begin{equation}\label{eq:directional-envelope}
 K_X(st)\leq \frac{V_X(t)s^2}{2(1-cs)},
 \qquad 0\leq s<1/c.
\end{equation}
The infimum of the empty set is infinity. If \(V_X(t)=0\), set \(C_X(t)=0\).
\end{definition}

When \(V_X(t)=0\), the centered projection \(\ip{t}{X}\) is zero almost surely, so the convention is exact.

\begin{theorem}[Exact variational formula]\label{thm:variational}
For \(V_X(t)>0\), extend \(M_t(s)\) by infinity whenever \(K_X(st)=\infty\), and use \(1/\infty=0\). Then
\begin{equation}\label{eq:variational}
 C_X(t)=\max\left\{0,\sup_{s>0}
 \frac{1-M_t(s)^{-1}}{s}\right\}.
\end{equation}
Moreover:
\begin{enumerate}
\item \(C_X(t)<\infty\) if and only if \(K_X(st)<\infty\) for some \(s>0\);
\item \(C_X(t)=0\) if and only if
\begin{equation}\label{eq:no-positive-jumps}
 \nu\{x:\ip{t}{x}>0\}=0;
\end{equation}
\item if \(C_X(t)<\infty\) and
\(\int|\ip{t}{x}|^3\,\nu(dx)<\infty\), then
\begin{equation}\label{eq:third-cumulant-bound}
 C_X(t)\geq \max\left\{0,\frac{\kappa_3(t)}{3V_X(t)}\right\},
 \qquad \kappa_3(t):=\int\ip{t}{x}^3\,\nu(dx).
\end{equation}
\end{enumerate}
\end{theorem}

\begin{proof}
Set
\[
 \gamma=\max\left\{0,\sup_{s>0}\frac{1-M_t(s)^{-1}}s\right\}.
\]
If \(s<1/\gamma\), then \(M_t(s)\) must be finite. Otherwise the expression in the supremum equals \(1/s>\gamma\). Since
\[
 \frac{1-M_t(s)^{-1}}s\leq\gamma,
\]
we have \(M_t(s)\leq(1-\gamma s)^{-1}\). The factorization proves that \(\gamma\) is feasible in \eqref{eq:directional-envelope}. Conversely, if \(c<\gamma\), some \(s>0\) satisfies
\((1-M_t(s)^{-1})/s>c\). Because the left-hand side is at most \(1/s\), this point obeys \(s<1/c\), and the envelope with scale \(c\) fails there. This proves \eqref{eq:variational}, including the case \(\gamma=\infty\).

Any finite feasible scale makes \(K_X(st)\) finite for all sufficiently small \(s>0\). For the converse, suppose \(K_X(s_0t)<\infty\). For \(z\geq0\),
\[
 \e^z-1-z-\frac{z^2}{2}\leq\frac{z^3\e^z}{6},
\]
whereas the left-hand side is nonpositive for \(z\leq0\). It follows that, uniformly for \(0<s\leq s_0/2\),
\[
 K_X(st)-\frac{V_X(t)s^2}{2}\leq Ls^3
\]
for some finite \(L\). Indeed, the positive-jump integral is controlled by the second moment near zero and by \(\e^{s_0\ip{t}{x}}\) on large positive jumps. Hence the positive part of
\((1-M_t(s)^{-1})/s\) is bounded near zero. Away from zero it is bounded by \(1/s\). Formula \eqref{eq:variational} is therefore finite, proving part 1.

If \eqref{eq:no-positive-jumps} holds, then
\(\e^z-1-z\leq z^2/2\) for every \(z\leq0\), and \eqref{eq:LK} gives
\(K_X(st)\leq V_X(t)s^2/2\) for all \(s\geq0\). Thus \(C_X(t)=0\). Conversely, suppose the projected L\'evy measure charges \((0,\infty)\). There are \(0<a<b<\infty\) for which
\[
 \delta:=\nu\{x:a\leq\ip{t}{x}\leq b\}>0.
\]
If a positive MGF is infinite, the variational formula already gives a positive pole. Otherwise, because \(\e^z-1-z\geq0\) for every real \(z\),
\[
 K_X(st)\geq\delta\bigl(\e^{sa}-1-sb\bigr).
\]
The right-hand side eventually exceeds \(V_X(t)s^2/2\), so \(M_t(s)>1\) for some \(s>0\). Formula \eqref{eq:variational} is then strictly positive. This proves part 2.

Under the assumptions of part 3, expansion at the origin gives
\[
 K_X(st)=\frac{V_X(t)s^2}{2}+\frac{\kappa_3(t)s^3}{6}+o(s^3),
 \qquad
 M_t(s)=1+\frac{\kappa_3(t)}{3V_X(t)}s+o(s).
\]
Taking \(s\downarrow0\) in \eqref{eq:variational} yields \eqref{eq:third-cumulant-bound}.
\end{proof}

\begin{remark}\label{rem:opposite-direction}
The left scale of \(\ip{t}{X}\) is \(C_X(-t)\). No symmetry is imposed, and no nonnegative-remainder order is used. Part 2 of Theorem~\ref{thm:variational} shows that a vanishing pole has an exact spectral meaning: there are no L\'evy jumps toward the relevant half-space.
\end{remark}

\begin{proposition}[Positive homogeneity]\label{prop:homogeneity}
For every \(a\geq0\) and \(t\in\R^d\),
\begin{equation}\label{eq:homogeneity}
 C_X(at)=aC_X(t).
\end{equation}
\end{proposition}

\begin{proof}
The case \(a=0\) is immediate. For \(a>0\), replace \(s\) by \(as\) in \eqref{eq:directional-envelope}, noting that \(V_X(at)=a^2V_X(t)\). A scale \(c\) is feasible for \(t\) exactly when \(ac\) is feasible for \(at\).
\end{proof}

Positive homogeneity allows the scalar minima to be characterized without choosing coordinates or a norm.

\begin{definition}[Admissible radial denominator]\label{def:admissible}
A function \(q:\R^d\to[0,\infty]\) is admissible for \(X\) if it is positively homogeneous and
\begin{equation}\label{eq:radial-envelope}
 K_X(t)\leq\frac{V_X(t)}{2(1-q(t))}
 \qquad\text{whenever }q(t)<1.
\end{equation}
\end{definition}

\begin{theorem}[Pointwise minimal radial envelope]\label{thm:minimal-radial}
The function \(C_X\) is admissible. If \(q\) is any admissible radial denominator, then
\begin{equation}\label{eq:pointwise-minimum}
 q(t)\geq C_X(t),\qquad t\in\R^d.
\end{equation}
Thus \(C_X\) is the pointwise least homogeneous pole compatible with the covariance quadratic form.
\end{theorem}

\begin{proof}
Positive homogeneity follows from Proposition~\ref{prop:homogeneity}. If \(C_X(t)<1\), put \(s=1\) in the defining directional envelope to obtain \eqref{eq:radial-envelope} with \(q=C_X\).

Now let \(q\) be admissible and fix \(t\). For every \(0\leq s<1/q(t)\), homogeneity gives \(q(st)=sq(t)<1\). Applying \eqref{eq:radial-envelope} at \(st\) yields
\[
 K_X(st)\leq\frac{s^2V_X(t)}{2(1-sq(t))}.
\]
Hence \(q(t)\) is a feasible directional scale, so it cannot be smaller than \(C_X(t)\).
\end{proof}

\begin{corollary}[Directional tails]\label{cor:tail}
If \(C_X(t)<\infty\), then for every \(x\geq0\),
\begin{equation}\label{eq:tail}
 \Pp\left\{\ip{t}{X}\geq
 \sqrt{2V_X(t)x}+C_X(t)x\right\}\leq\e^{-x}.
\end{equation}
\end{corollary}

\begin{proof}
For \(V_X(t)>0\), use the exponential Markov inequality with
\[
 s=\frac{\sqrt{2x/V_X(t)}}{1+C_X(t)\sqrt{2x/V_X(t)}}.
\]
The exponent given by \eqref{eq:directional-envelope} is at most \(-x\). If \(V_X(t)=0\), the projection is zero almost surely.
\end{proof}

\begin{corollary}[Best Euclidean pole]\label{cor:euclidean}
Let
\begin{equation}\label{eq:global-pole}
 \overline C_X=\sup_{\|u\|=1}C_X(u).
\end{equation}
Then
\begin{equation}\label{eq:covariance-isotropic}
 K_X(t)\leq\frac{t^\top\Sigma t}{2(1-\overline C_X\|t\|)},
 \qquad \overline C_X\|t\|<1.
\end{equation}
Moreover, \(\overline C_X\) is the smallest constant that can replace it in \eqref{eq:covariance-isotropic}.
\end{corollary}

\begin{proof}
The denominator \(q(t)=\overline C_X\|t\|\) dominates \(C_X(t)\) by homogeneity and is therefore admissible. Conversely, any constant in \eqref{eq:covariance-isotropic} defines an admissible denominator of this form, so Theorem~\ref{thm:minimal-radial} forces it to dominate \(C_X(u)\) for every unit \(u\).
\end{proof}

Corollary~\ref{cor:euclidean} connects the directional theory to a one-constant vector inequality. Unlike an isotropic variance proxy, however, the numerator retains the exact covariance quadratic form.

\section{Geometry and structural properties}\label{sec:geometry}

We first record operations under which the radial pole behaves exactly or monotonically.

\begin{proposition}[Linear maps, sums, and L\'evy time]\label{prop:structure}
Let all random vectors below be centered and have finite second moments.
\begin{enumerate}
\item If \(L:\R^d\to\R^m\) is linear, then
\begin{equation}\label{eq:linear-map}
 C_{LX}(t)=C_X(L^\top t),\qquad t\in\R^m.
\end{equation}
\item If \(X\) and \(Y\) are independent vectors in \(\R^d\), then
\begin{equation}\label{eq:convolution}
 C_{X+Y}(t)\leq\max\{C_X(t),C_Y(t)\}.
\end{equation}
\item If \((L_r)_{r\geq0}\) is a centered L\'evy process with finite second moments, then for every \(r>0\),
\begin{equation}\label{eq:levy-time}
 C_{L_r}(t)=C_{L_1}(t).
\end{equation}
\end{enumerate}
\end{proposition}

\begin{proof}
The identity \(\ip{t}{LX}=\ip{L^\top t}{X}\) proves part 1. For part 2, set
\[
 c=\max\{C_X(t),C_Y(t)\}.
\]
If \(c=\infty\), the claim is immediate. Assume \(c<\infty\).
Independence and covariance additivity give
\begin{align*}
 K_{X+Y}(st)
 &=K_X(st)+K_Y(st)\\
 &\leq\frac{s^2V_X(t)}{2(1-cs)}+
       \frac{s^2V_Y(t)}{2(1-cs)}
 =\frac{s^2V_{X+Y}(t)}{2(1-cs)}.
\end{align*}
Part 3 follows from \(K_{L_r}=rK_{L_1}\) and \(V_{L_r}=rV_{L_1}\): the normalized remainder MGF in \eqref{eq:factorization} does not depend on \(r\).
\end{proof}

The scale \(C_X\) alone need not determine the law. The richer remainder family does.

\begin{proposition}[Reconstruction from directional remainders]\label{prop:reconstruction}
Suppose that \(V_X(u)\) and \(\Law(R_u)\) are known for every unit vector \(u\) with \(V_X(u)>0\), and that directions with zero variance are identified. Then the law of \(X\), and hence its L\'evy triplet, is determined.
\end{proposition}

\begin{proof}
For \(\theta\neq0\), write \(\theta=ru\) with \(u\) unit and \(r\in\R\). If \(V_X(u)>0\), formula \eqref{eq:imaginary-factorization} gives the characteristic exponent \(\Psi_X(ru)\) from \(V_X(u)\) and the characteristic function of \(R_u\). If \(V_X(u)=0\), the projection is zero and the exponent is zero on that line. Thus the characteristic function of \(X\) is known on \(\R^d\). Uniqueness of the characteristic function and of the L\'evy--Khintchine triplet completes the proof \citep{sato2013}.
\end{proof}

We next compare \(C_X\) with two familiar positively homogeneous functions. Let
\begin{equation}\label{eq:effective-domain}
 \mathcal E_X=\{t\in\R^d:K_X(t)<\infty\},
 \qquad D_X=\interior\mathcal E_X.
\end{equation}
The effective domain is convex by H\"older's inequality. Assume in this subsection that \(0\in D_X\). The open convex set \(D_X\) has Minkowski functional
\begin{equation}\label{eq:gauge}
 p_{D_X}(t)=\inf\{a>0:t\in aD_X\}.
\end{equation}
We also define the positive jump radius
\begin{equation}\label{eq:support-function}
 b_X(t)=\sup_{x\in\supp\nu}\ip{t}{x}_+\in[0,\infty],
 \qquad z_+=\max\{z,0\}.
\end{equation}
It is the support function of \(\supp\nu\cup\{0\}\), with extended values allowed.

\begin{lemma}[A one-sided exponential remainder]\label{lem:exp-bound}
For every real \(z<3\),
\begin{equation}\label{eq:exp-bound}
 \e^z-1-z\leq\frac{z^2}{2(1-z_+/3)}.
\end{equation}
\end{lemma}

\begin{proof}
For \(z\leq0\), the function \(z^2/2-(\e^z-1-z)\) is nonnegative. For \(0\leq z<3\), compare power-series coefficients:
\[
 \e^z-1-z=\sum_{n\geq2}\frac{z^n}{n!}
 \leq\frac{z^2}{2}\sum_{k\geq0}\left(\frac z3\right)^k,
\]
because \(n!\geq2\,3^{n-2}\) for \(n\geq2\).
\end{proof}

\begin{theorem}[Domain and support sandwich]\label{thm:sandwich}
If \(0\in D_X\), then for every \(t\in\R^d\),
\begin{equation}\label{eq:sandwich}
 p_{D_X}(t)\leq C_X(t)\leq\frac{b_X(t)}{3}.
\end{equation}
The factor \(1/3\) in the upper bound is best possible over centered infinitely divisible laws with bounded positive jumps.
\end{theorem}

\begin{proof}
Fix \(t\) and let
\[
 \tau(t)=\sup\{s\geq0:K_X(st)<\infty\}.
\]
Since \(0\in D_X\), convexity gives \(p_{D_X}(t)=1/\tau(t)\), with the usual extended conventions. If a directional scale \(c\) were smaller than \(p_{D_X}(t)\), then \(1/c>\tau(t)\). The envelope would require \(K_X(st)<\infty\) at some \(s\in(\tau(t),1/c)\), a contradiction. Taking the least scale proves the lower bound.

The upper bound is immediate when \(b_X(t)=\infty\). Suppose \(b=b_X(t)<\infty\), and take \(0\leq s<3/b\). For \(y=\ip{t}{x}\), Lemma~\ref{lem:exp-bound} and \(y_+\leq b\) yield
\[
 \e^{sy}-1-sy\leq\frac{s^2y^2}{2(1-sb/3)}.
\]
The Gaussian term in \eqref{eq:LK} is bounded by the same denominator. Integration gives
\[
 K_X(st)\leq\frac{V_X(t)s^2}{2(1-sb/3)},
\]
so \(C_X(t)\leq b/3\).

For sharpness, let \(N\) be Poisson with mean \(\lambda>0\) and put \(Y=a(N-\lambda)\), \(a>0\). Its positive jump radius is \(a\), while \(V=\lambda a^2\) and \(\kappa_3=\lambda a^3\). Theorem~\ref{thm:variational} gives \(C_Y\geq a/3\), and the upper bound gives the reverse inequality.
\end{proof}

The lower function in \eqref{eq:sandwich} is sublinear because it is the gauge of a convex set \citep{rockafellar1970}. The upper function is sublinear because it is a support function. The radial pole lies between them, but it need not inherit their convexity; Section~\ref{sec:regularity} gives an exact counterexample.

The zero directions have a simple cone description. For a cone \(K\subset\R^d\), write
\[
 K^\circ=\{t:\ip{t}{x}\leq0\text{ for every }x\in K\}.
\]

\begin{corollary}[Zero cone]\label{cor:zero-cone}
For every centered infinitely divisible vector with finite second moment,
\begin{equation}\label{eq:zero-cone}
 \{t:C_X(t)=0\}
 =\{t:b_X(t)=0\}
 =\bigl(\overline{\cone}(\supp\nu)\bigr)^\circ.
\end{equation}
\end{corollary}

\begin{proof}
Theorem~\ref{thm:variational} identifies \(C_X(t)=0\) with the absence of positive projected jumps. This is equivalent to \(\ip{t}{x}\leq0\) for all \(x\in\supp\nu\), which is the second equality in \eqref{eq:zero-cone} and is unchanged under conic hull and closure.
\end{proof}

For a purely Gaussian vector, the L\'evy support is empty, the polar cone is all of \(\R^d\), and \(C_X\equiv0\). For a jump measure whose closed conic hull is \(\R^d\), the zero cone reduces to the origin.

\section{Regularity and nonconvexity}\label{sec:regularity}

Positive homogeneity does not imply subadditivity. The following example uses only two L\'evy atoms and gives every scale in closed form.

\begin{proposition}[The radial pole need not be subadditive]\label{prop:non-subadditive}
Let \(N_1,N_2\) be independent Poisson variables with mean one, put
\[
 x_1=(-2,-2),\qquad x_2=(-2,-1),
\]
and define
\begin{equation}\label{eq:cp-vector}
 X=N_1x_1+N_2x_2-(x_1+x_2).
\end{equation}
Then \(\Sigma\) is positive definite, but for
\[
 t_1=(-2,1),\qquad t_2=(-1,2),
\]
one has
\begin{equation}\label{eq:non-subadditive-values}
 C_X(t_1)=\frac{35}{39},\qquad
 C_X(t_2)=0,
 \qquad C_X(t_1+t_2)=1.
\end{equation}
Consequently,
\begin{equation}\label{eq:non-subadditive}
 C_X(t_1+t_2)>C_X(t_1)+C_X(t_2).
\end{equation}
\end{proposition}

\begin{proof}
The vectors \(x_1,x_2\) span \(\R^2\), so
\(\Sigma=x_1x_1^\top+x_2x_2^\top\) is positive definite. In direction \(t_1\), the two jump sizes are \(2\) and \(3\). Hence
\[
 K_X(st_1)=\sum_{n\geq2}\frac{(2^n+3^n)s^n}{n!},
 \qquad V_X(t_1)=13,
 \qquad \kappa_3(t_1)=35.
\]
The cubic lower bound gives \(C_X(t_1)\geq35/39\). Set \(c=35/39\),
\[
 A_n=2^n+3^n,
 \qquad B_n=\frac{13n!}{2}c^{n-2}.
\]
Then \(A_2=B_2=13\) and \(A_3=B_3=35\). For \(n\geq3\),
\[
 A_{n+1}<3A_n,
 \qquad \frac{B_{n+1}}{B_n}=(n+1)c>3.
\]
Induction gives \(A_n\leq B_n\) for all \(n\geq2\). Comparing the nonnegative coefficients proves
\[
 K_X(st_1)\leq\frac{13s^2}{2(1-cs)},
\]
and therefore \(C_X(t_1)=35/39\).

In direction \(t_2\), the jumps are \(-2\) and \(0\), so Theorem~\ref{thm:variational} gives \(C_X(t_2)=0\). In direction \(t_1+t_2=(-3,3)\), they are \(0\) and \(3\). The cubic lower bound and Theorem~\ref{thm:sandwich} both equal one, which proves the last value in \eqref{eq:non-subadditive-values} and hence \eqref{eq:non-subadditive}.
\end{proof}

The next result gives a clean sufficient condition for directional continuity. It is stated on the unit sphere \(\mathbb S^{d-1}\); positive homogeneity then controls all nonzero points.

\begin{theorem}[Continuity under global exponential moments]\label{thm:continuity}
Assume that \(\Sigma\) is positive definite and
\begin{equation}\label{eq:all-exponential-moments}
 K_X(t)<\infty\qquad\text{for every }t\in\R^d.
\end{equation}
Then \(u\mapsto C_X(u)\) is finite and continuous on \(\mathbb S^{d-1}\).
\end{theorem}

\begin{proof}
For \(u\in\mathbb S^{d-1}\) and \(s>0\), set
\[
 M(u,s)=\frac{2K_X(su)}{V_X(u)s^2},
 \qquad
 F(u,s)=\frac{1-M(u,s)^{-1}}s.
\]
Assumption \eqref{eq:all-exponential-moments} permits differentiation of the L\'evy integral on every compact subset of \(\R^d\), so \(K_X\) is smooth. Since \(\Sigma\) is positive definite, \(V_X(u)\) is bounded away from zero on the sphere. Taylor expansion, uniformly for \(u\) on the sphere, therefore extends the two functions continuously to \(s=0\) by
\begin{equation}\label{eq:F-zero}
 M(u,0)=1,
 \qquad F(u,0)=\frac{\kappa_3(u)}{3V_X(u)}.
\end{equation}
For \(T>0\), define
\[
 G_T(u)=\max\left\{0,\max_{0\leq s\leq T}F(u,s)\right\}.
\]
The maximum of a continuous function over the fixed compact interval \([0,T]\) is continuous in \(u\). Also \(M(u,s)>0\), so \(F(u,s)\leq1/s\). The variational formula gives
\begin{equation}\label{eq:uniform-approximation}
 0\leq C_X(u)-G_T(u)\leq\frac1T,
 \qquad u\in\mathbb S^{d-1}.
\end{equation}
Thus \(G_T\) converges uniformly to \(C_X\) as \(T\to\infty\), proving continuity and finiteness.
\end{proof}

The global moment assumption is stronger than pointwise finiteness near a selected direction. Without it, even local boundedness can fail.

\begin{proposition}[Finite variance does not imply local boundedness]\label{prop:discontinuity}
There is a centered compound Poisson vector in \(\R^2\) with finite second moment and positive-definite covariance such that \(C_X(u_0)=0\) for one unit vector \(u_0\), while \(C_X(u_n)=\infty\) along unit vectors \(u_n\to u_0\).
\end{proposition}

\begin{proof}
Take no Gaussian component and the finite L\'evy measure
\begin{equation}\label{eq:discontinuous-levy}
 \nu=\sum_{n\geq1}n^{-8}\delta_{(-n,n^2)},
\end{equation}
with the drift chosen to center the law. It has finite second moment because
\[
 \sum_{n\geq1}n^{-8}(n^2+n^4)<\infty.
\]
Its support contains two linearly independent vectors, so its covariance is positive definite.

Let \(u_0=(1,0)\). Every projected jump is \(-n<0\), hence \(C_X(u_0)=0\). For \(\varepsilon>0\), put
\[
 u_\varepsilon=\frac{(1,\varepsilon)}{\sqrt{1+\varepsilon^2}}.
\]
The \(n\)th projected jump is
\[
 y_{n,\varepsilon}=\frac{-n+\varepsilon n^2}{\sqrt{1+\varepsilon^2}}.
\]
For all sufficiently large \(n\), it is positive and grows quadratically. Consequently, for every \(s>0\), the terms
\(n^{-8}\e^{s y_{n,\varepsilon}}\) do not tend to zero, and the positive part of the L\'evy integral diverges. Thus \(K_X(su_\varepsilon)=\infty\) for every \(s>0\). Part 1 of Theorem~\ref{thm:variational} gives \(C_X(u_\varepsilon)=\infty\). Taking \(\varepsilon\downarrow0\) proves the claim.
\end{proof}

Proposition~\ref{prop:discontinuity} also explains why extended values are part of the definition rather than a technical afterthought. Finite covariance controls quadratic behavior at the origin but does not supply a common positive neighborhood for directional MGFs.

\section{Exact gamma-ray models}\label{sec:gamma}

The geometric lower bound becomes exact for a broad finite-ray model. Let \(Z\) be a centered Gaussian vector in \(\R^d\) with covariance \(A\). For \(j=1,\ldots,m\), let \(G_j\) be independent gamma variables with shape \(\alpha_j>0\) and scale \(\beta_j>0\), independent also of \(Z\), and let \(v_j\in\R^d\). Consider
\begin{equation}\label{eq:gamma-ray-model}
 X=Z+\sum_{j=1}^m v_j(G_j-\alpha_j\beta_j).
\end{equation}

\begin{theorem}[Gamma-ray formula]\label{thm:gamma-rays}
For the model \eqref{eq:gamma-ray-model},
\begin{equation}\label{eq:gamma-cumulant}
 K_X(t)=\frac12t^\top A t+
 \sum_{j=1}^m\alpha_j\left[-\log(1-\beta_j\ip{t}{v_j})-
 \beta_j\ip{t}{v_j}\right]
\end{equation}
on the domain
\begin{equation}\label{eq:gamma-domain}
 D_X=\{t:\beta_j\ip{t}{v_j}<1\text{ for }j=1,\ldots,m\}.
\end{equation}
For every \(t\in\R^d\),
\begin{equation}\label{eq:gamma-pole}
 C_X(t)=\max_{1\leq j\leq m}
 \bigl(\beta_j\ip{t}{v_j}\bigr)_+
 =p_{D_X}(t).
\end{equation}
Equivalently, if
\begin{equation}\label{eq:gamma-polytope}
 P=\conv\{0,\beta_1v_1,\ldots,\beta_mv_m\},
\end{equation}
then \(C_X=h_P\), the support function of \(P\).
\end{theorem}

\begin{proof}
The gamma MGF gives \eqref{eq:gamma-cumulant} and \eqref{eq:gamma-domain}. Put
\[
 a_j=\beta_j\ip{t}{v_j},
 \qquad c=\max_j(a_j)_+.
\]
For \(0\leq z<1\),
\begin{equation}\label{eq:gamma-positive-bound}
 -\log(1-z)-z=\sum_{n\geq2}\frac{z^n}{n}
 \leq\frac{z^2}{2(1-z)}.
\end{equation}
For \(z\leq0\),
\begin{equation}\label{eq:gamma-negative-bound}
 -\log(1-z)-z\leq\frac{z^2}{2};
\end{equation}
the difference between the right and left sides has derivative
\(-z^2/(1-z)\) and vanishes at zero. Therefore, for \(0\leq s<1/c\), every positive \(a_j\) is bounded using \eqref{eq:gamma-positive-bound} and \(1-sa_j\geq1-sc\), while every nonpositive \(a_j\) is bounded using \eqref{eq:gamma-negative-bound}. The Gaussian term is also dominated after inserting the denominator. Since
\[
 V_X(t)=t^\top A t+\sum_{j=1}^m\alpha_ja_j^2,
\]
we obtain
\[
 K_X(st)\leq\frac{V_X(t)s^2}{2(1-cs)}.
\]
Thus \(C_X(t)\leq c\).

The intersection of the \(m\) half-spaces in \eqref{eq:gamma-domain} has gauge
\[
 p_{D_X}(t)=\max_j(a_j)_+=c.
\]
The lower bound in Theorem~\ref{thm:sandwich} gives \(C_X(t)\geq c\), proving \eqref{eq:gamma-pole}. Finally, the maximum of the linear functionals \(\ip{t}{\beta_jv_j}\) and zero is the support function of the convex hull in \eqref{eq:gamma-polytope}.
\end{proof}

\begin{remark}\label{rem:gamma-insight}
The shapes \(\alpha_j\) and the Gaussian covariance affect the true variance in the numerator but not the pole. The pole is set by the first positive gamma singularity along each ray. The support function in \eqref{eq:gamma-polytope} belongs to the finite rate-vector polytope; it is not the support function of the unbounded L\'evy jump support.
\end{remark}

For \(m=1\), Theorem~\ref{thm:gamma-rays} recovers the familiar gamma scale in the positive direction and zero in every direction that sees only negative jumps. With several noncollinear rays, it produces a piecewise-linear, generally asymmetric pole on \(\R^d\). This is a case where the radial minimum is a convex gauge, in contrast with Proposition~\ref{prop:non-subadditive}.

\section{Discussion}\label{sec:discussion}

Fixing the quadratic proxy at \(t^\top\Sigma t\) separates local variance from the global obstruction to exponential concentration. For infinitely divisible vectors, that obstruction is directional. The function \(C_X\) records it without replacing the covariance by an isotropic upper bound. Corollary~\ref{cor:euclidean} recovers the best one-constant denominator when such a summary is needed, while the full function retains directions with a zero pole, a finite pole, or no positive exponential moment.

The two geometric bounds play different roles. The domain gauge is unavoidable: any Bernstein denominator must place its pole no farther than the MGF boundary. The positive jump radius gives a universal sufficient scale when projected positive jumps are bounded. Neither bound characterizes the optimum in every law, and the two-atom example shows why \(C_X\) cannot generally be treated as a norm or convex body gauge. Gamma rays supply a complementary exact case: the domain bound is attained and can be written as the support function of a finite polytope even though the underlying positive jump support is unbounded.

The regularity theorem isolates the assumption that makes the directional optimization stable. Global exponential moments turn the variational problem into a uniform limit of compact maximizations. Finite variance only fixes the second-order expansion and cannot prevent the positive MGF domain from collapsing under an arbitrarily small rotation. This distinction matters whenever a global vector constant is formed by taking the supremum over directions.

Statistical estimation of \(C_X\) is a natural next question. The canonical measure already has an inference literature \citep{watteel2003}, but a direction-indexed pole raises uniformity and identifiability problems, especially when infinite values are possible. A second question is to find structural conditions, between bounded positive jumps and finite gamma rays, that force subadditivity. The domain sandwich and the two counterexamples delimit both problems.

\begin{acknowledgement}[title={Declarations}]
\textbf{Funding.} The authors received no specific funding for this work.

\textbf{Competing interests.} The authors declare no competing interests.

\textbf{Data availability.} No datasets were generated or analyzed. All examples are analytic.

\textbf{Author contributions.} Y.C. developed the theory, proofs, counterexamples, and initial manuscript. X.W. supervised the project and reviewed the mathematical argument and manuscript.

\textbf{Related manuscripts.} Two unpublished manuscripts by the authors treat the preceding scalar cases. The present article is self-contained and separates its multivariate results in the Introduction.
\end{acknowledgement}

\bibliographystyle{bib/vmsta2-mathphys}
\bibliography{bib/references}

\end{document}